\documentclass [12pt,a4paper]{article}
\usepackage{amssymb} 
\usepackage{graphicx}
\usepackage{graphics}
\usepackage{amsmath}
\usepackage{amsthm}
\newtheorem*{thm}{Theorem}

\newtheorem*{lemma}{Lemma}
\newtheorem*{ex}{Example}

\newtheorem*{defi}{Definition}

\def\qed{\nobreak\hfill $\square$}

\def\<{\langle}
\def\>{\rangle}

\def\Mn{M_n(\bbbc)}
\def\Mnsa{M_n^{sa}(\bbbc)}
\def\Mnp{M_n^{+}(\bbbc)}
\def\Mr{M_r(\bbbc)}

\def\M3{M_3(\bbbc)}
\def\M3r{M_3(\bbbr)}

\def\diag{\mathrm{Diag}}

\def\im{\mathrm{i}}

\def\bbbr{{\mathbb R}}
\def\bbbc{{\mathbb C}}

\def\Tr{\mathrm{Tr}\,}

\def\var{{\rm Var}}

\newcommand{\R}{\mathbb{R}}

\newcommand{\C}{\mathbb{C}}

\newcommand*{\ba}{\begin{array}}
\newcommand*{\ea}{\end{array}}
\newcommand*{\be}{\begin{equation}}
\newcommand*{\ee}{\end{equation}}
\newcommand{\tr}{\mathrm{Tr}}

\newcommand{\mc}[1]{\mathcal{#1}}
\newcommand{\conv}[1]{\mathbf{Conv}\left(#1\right)}

\newcommand{\Kk}{\mathcal{K}}

\newcommand{\norm}[1]{\left\|#1\right\|}

\newcommand*{\inner}[2]{\left<#1,\,#2\right>}

\begin{document}
\bigskip
\centerline{\LARGE {\bf A characterization theorem for matrix variances}}
\bigskip
\centerline{{\bf D\'enes Petz\footnote{E-mail: petz@math.bme.hu } and D\'aniel 
Virosztek\footnote{E-mail: virosz89@gmail.com}}}

\begin{center}
Department of Mathematical Analysis, \\ Budapest University of Technology and Economics, \\ 
Egry J\'ozsef u.~1., Budapest, 1111 Hungary
\end{center}

\bigskip

\begin{abstract}
Some recent papers formulated sufficient conditions for the decomposition of matrix variances \cite{LZ-PD, PD-TG}. A statement was that if we have one or two observables, then the decomposition is possible. In this paper we consider an arbitrary finite set of observables and we present a necessary and sufficient condition for the decomposition of the matrix variances.
\end{abstract}

The subject here is matrix theory, see \cite{Bh2, HP, OP}. By a density matrix $D \in \Mn$ we 
mean $D \ge 0$ and $\Tr D =1$. In quantum information theory the traditional variance is defined by
$$
\var_D (A)=\Tr D A^2 - (\Tr D A)^2,
$$
where  $D$ is a density matrix and $A \in \Mn$ is a self-adjoint operator. This noncommutative variance is a natural extension of the variance in probability theory \cite{Fell}, and has several applications \cite{GHP, HP, Hol, PDco, PDq, PD-TG}.
It is easy to show that
$$
\var_D (A+\lambda I)=\var_D (A) \quad \quad (\lambda \in \bbbr)
$$
and the concavity of the variance functional is well-known:
$$
\var_{D}(A) \ge \sum_i \lambda_i \var_{D_i}(A)\,, 
$$ 
when $D=\sum_i \lambda_i D_i$, $\lambda_i \ge 0$ and $\sum_i \lambda_i=1$. It was proved 
in \cite{PD-TG} that for every  self-adjoint operator $A$ and density matrix $D$ there are projections $P_k$ such that $D= \sum_k \lambda_k P_k$ with $0 < \lambda_k$, $\sum_k \lambda_k=1$ and $\var_{D}(A) = \sum_k \lambda_k \var_{P_k}(A)$ holds. 

There is another example when  the previous $A$ is replaced with $A_1, A_2, \dots, A_r$, 
they are also self-adjoint operators. Then the standard variance is a matrix in $\Mr$:
$$
\left[\var_{D} (A_1, \dots, A_r) \right]_{i,j}=\Tr D A_i A_j -(\Tr D A_i) (\tr D A_j) 
\qquad (1\le i,j \le r).
$$

Assume that $0 \le \lambda_1, \lambda_2$ and $\lambda_1+ \lambda_2=1$. An elementary
computation gives that
$$
[\var_{\lambda_1 D_1+\lambda_2 D_2} (A_1,\dots, A_r)-\lambda_1\var_{D_1} (A_1,\dots, A_r)-
\lambda_2\var_{D_2} (A_1,\dots ,A_r)]_{ij}
= \lambda_1 \lambda_2 a_i a_j , \quad
$$
where
$$
a_i=\Tr( D_{1}- D_{2}) A_i , \qquad 1 \le i \le r.
$$
It follows that
$$
\var_{\lambda_1 D_1+\lambda_2 D_2} (A_1,\dots, A_r) \ge \lambda_1\var_{D_1} (A_1,\dots, A_r)+
\lambda_2\var_{D_2} (A_1,\dots ,A_r)\, .
$$

So we have the concavity of the variance functional $D \mapsto \var_{D}(A_1,\dots, A_r)$:
$$
\var_{D}(A_1,\dots, A_r) \ge \sum_i \lambda_i \var_{D_i}(A_1,\dots, A_r)\quad \mbox{if} 
\quad D=\sum_i \lambda_i D_i,
$$
where $\lambda_i \ge 0$ and $\sum_i \lambda_i=1$. For $r=2$ the equality may be also true and this
is the result in \cite{LZ-PD}: $D$ is a certain convex combination of projections $P_k$ 
as $D=\sum_k p_kP_k$ and
$$
\var_D (A_1,A_2) = \sum_k p_k \var_{P_k} (A_1,A_2).
$$

In the present paper we give a necessary and sufficient condition for the previous equality 
for an arbitrary set $\{A_1,\dots, A_r\}$ of self-adjoint operators.


\section{General computations}

As it was declared, the variance functional is concave, that is, if $D_1, \dots, D_m$ are density 
matrices, $A_1, \dots, A_r$ are self-adjoint operators and $D=\sum_{k=1}^{m} \lambda_k D_k$ with 
some $0 \leq \lambda_1, \dots, \lambda_m,$ $\sum_{k=1}^{m} \lambda_k =1$ then
\be     \label{konkav}
\var_{D} (A_1, \dots, A_r) \geq \sum_{k=1}^{m} \lambda_k \var_{D_k} (A_1, \dots, A_r).
\ee
We are interested in the case of equality in (\ref{konkav}).

\begin{lemma} 
If $D_1, \dots, D_m$ are density matrices, $A_1, \dots, A_r$ are self-adjoint operators and 
$D=\sum_{k=1}^{m} \lambda_k D_k$ with some $0 < \lambda_1, \dots, \lambda_m,$ $\sum_{k=1}^{m} 
\lambda_k =1,$ then
\be
\label{equal}
\var_{D} (A_1, \dots, A_r) = \sum_{k=1}^{m} \lambda_k \var_{D_k} (A_1, \dots, A_r)
\ee
if and only if 
\be \label{trace}
\tr D_k A_j =\tr D A_j \quad {\rm for \,\, all} \quad 1\le k \le m \quad{\rm and} 
\quad  1 \le j \le r.
\ee
\end{lemma}

\begin{proof}
The variance has the shift invariance property 
$$\var_{D} (A_1, \dots, A_r) = \var_{D} (A_1-\lambda_1 I, \dots, A_r-\lambda_r I)$$
for every reals $\lambda_1, \dots, \lambda_r.$ Set $\lambda_j := \tr D A_j.$ With this choice 
$\tr D (A_j-\lambda_j I)=0$ holds for every j. Therefore
$$[\var_{D} (A_1, \dots, A_r)]_{ij} = [\var_{D} (A_1-\lambda_1 I, \dots, A_r-\lambda_r I)]_{ij}= 
\tr D (A_i-\lambda_i I) (A_j-\lambda_j I).$$
Because of the concavity ot the variance,
$$
\var_{D} (A_1, \dots, A_r)- \sum_{k=1}^{m} \lambda_k \var_{D_k} (A_1, \dots, A_r)
$$
is a positive semi-definite matrix, hence it is equal to zero if and only if the diagonal 
elements are zeros, that is,
\be \label{szam}
\tr D (A_j-\lambda_j I)^2-\left(\sum_{k=1}^{m} \lambda_k \tr D_k (A_j-\lambda_j I)^2- \lambda_k 
\left(\Tr D_k (A_j-\lambda_j I)\right)^2\right)=0
\ee
holds for every $j.$ It is easy to check that (\ref{szam}) holds if and only if $\tr D_k 
(A_j-\lambda_j I)=0$ for every $k$, $j$ and this is equivalent to (\ref{trace}).
\end{proof} \qed

In the next section we use this Lemma to characterize those sets of self-adjoint operators for which the decomposition of the matrix variances with projections is possible.


\section{The main theorem}

Let us introduce some notations:
\[
\Mnsa:=\{ A \in \Mn: \ A^{*}=A\}\,,
\]
\[
\Mnp:=\{ C \in \Mn: \ C \geq 0\}\,,
\]
\[
\mc{S}(\C^{n}):=\{ D \in \Mn: \ D \geq 0, \ \tr D=1 \}\, .
\]
For an arbitrary subspace $\Kk \subset \C^n, $ we denote by $Q^{\Kk}$ the orthogonal 
projection onto $\Kk.$ We define
\[
A^{\Kk}:=Q^{\Kk} A Q^{\Kk}
\]
for every operator $A \in \Mn$ and
\[ 
\mc{B}(\Kk):=Q^{\Kk} \Mn Q^{\Kk}, \qquad \mc{B}^{sa}(\Kk):=Q^{\Kk} \Mnsa Q^{\Kk}, 
\]
\[
\mc{B}^{+}(\Kk):=Q^{\Kk} \Mnp Q^{\Kk}, \qquad  \mc{S}(\Kk):=\{X \in B^{+}(\Kk): \ \tr X=1\}.
\]
\begin{defi}
Let $\{A_1, \dots, A_r\}$ be a set of self-adjoint operators in $\Mn.$
$\{A_1, \dots, A_r\}$ is said to be \emph{variance-decomposable} if for every $D \in \mc{S}(\C^n)$ 
there exist $P_1, \dots, P_m$ rank-one projections such that
$$
D=\sum_{k=1}^{m} \lambda_k P_k \qquad \rm{ and } \qquad \var_{D} (A_1, \dots, A_r) = 
\sum_{k=1}^{m} \lambda_k \var_{P_k} (A_1, \dots, A_r)
$$
with some $0 \leq \lambda_k, \ \sum_k \lambda_k=1.$
\end{defi}

\begin{thm}
$\{A_1, \dots, A_r\} \subset \Mn$ is variance-decomposable if and only if
\be \label{tet}
\mathrm{dim}\left( \mathrm{span}\left\{ I^{\Kk}, A_1^{\Kk}, \dots, A_r^{\Kk}\right\}\right) < 
\left( \mathrm{dim}\,\, \Kk\right)^2
\ee
for every $\Kk \subset \C^n$ subspace with $ \mathrm{dim}\, \Kk > 1$.
\end{thm}

Note that this theorem immediately shows that every set of self-adjoint operators with at most 
two elements is variance-decomposable. (This is the result of \cite{PD-TG} and \cite{LZ-PD}.)

\begin{proof}
By the Lemma, $\{A_1, \dots, A_r\} \subset \Mn$ is variance-decomposable if and only 
if for every $D \in \mc{S}(\C^n)$ density operator there exist $P_1, \dots, P_m$ rank-one 
projections such that $D \in \conv{\{P_1, \dots, P_m\}}$ -- where $\conv{H}$ denotes the 
convex hull of the set H -- and
\be
\tr P_k A_j= \tr D A_j \,\, \rm{ for \,\, all } \,\, 1 \leq k \leq m \ \rm{ and } \ 1 
\leq j \leq r. 
\ee
We show that the condition (\ref{tet}) is sufficient.
It is enough to show that for every $D \in \mc{S}(\C^n), \ \mathrm{rank}(D)>1$ there exist 
$E_1, \dots, E_m \in \mc{S}(\C^n)$ density operators such that
\be \label{konv}
D \in \conv{\{E_1, \dots, E_m\}}
\ee
and
\be \label{linr}
\Tr E_k A_j= \tr D A_j \,\, \rm{ for \,\, all } \,\, k \,\, \rm{ and } \,\, j 
\ee
and
\be \label{rang}
\mathrm{rank}(E_k) < \mathrm{rank}(D).
\ee

Let $D$ be an arbitrary element of  $\mc{S}(\C^n)$ with $\mathrm{rank}(D)>1, \ \Kk:= 
\mathrm{range}(D).$
$\mc{B}^{sa}(\Kk)$ is a $\left(\mathrm{dim}(\Kk)\right)^2$ dimensional Hilbert space over 
the field $\R$ with the positive definite inner product $\inner{X}{Y}=\tr X Y.$
Let us use the notation $\mathbf{A}=(A_1, \dots, A_r).$ Define
$$
\mc{L}_{D, \mathbf{A}}^{\Kk}:=\{X \in \mc{B}^{sa}(\Kk): \ \inner{X}{I}=1, \inner{X}{A_j}=
\inner{D}{A_j} \,\, \rm{ for \,\, all } \,\, j\}.
$$
Clearly,
$$
\mc{L}_{D, \mathbf{A}}^{\Kk}=\{X \in \mc{B}^{sa}(\Kk): \ \inner{X}{I^{\Kk}}=1, \inner{X}{A_j^{\Kk}}=
\inner{D}{A_j^{\Kk}} \,\, \rm{ for\,\, all }\,\, j\}.
$$
Because of the assumption $\mathrm{dim}\left( \mathrm{span}\left\{ I^{\Kk}, A_1^{\Kk}, \dots, 
A_r^{\Kk}\right\}\right) < \left( \mathrm{dim} \Kk\right)^2,$ $\mc{L}_{D, \mathbf{A}}^{\Kk}$ is an 
affine subspace of $\mc{B}^{sa}(\Kk)$ with positive dimension. \par
It is well-known that $\mc{S}(\Kk)$ is a bounded convex set (for example, $\norm{P}_{2} \leq 1$ 
if $P \in \mc{S}(\Kk),$ where $\norm{\cdot}_2$ denotes the Hilbert-Schmidt norm). Therefore, 
$\mc{L}_{D, \mathbf{A}}^{\Kk} \cap \mc{S}(\Kk)$ is a bounded  convex set and
$$
\mc{L}_{D, \mathbf{A}}^{\Kk} \cap \mc{S}(\Kk) \subset \conv{\mc{L}_{D, \mathbf{A}}^{\Kk} \cap  
\partial \mc{S}(\Kk)},
$$
where $\partial \mc{S}(\Kk)$ denotes the relative boundary of $\mc{S}(\Kk).$
\par
By definition, $D \in \mc{L}_{D, \mathbf{A}}^{\Kk} \cap \mc{S}(\Kk),$ and hence
$$
D=\sum_{k=1}^{m} \lambda_k E_k \,\,\rm{ with \,\, some } \,\, \{E_k\}_{k=1}^{m} 
\subset \mc{L}_{D, \mathbf{A}}^{\Kk} \cap  \partial \mc{S}(\Kk) \,\, \rm{ and } \,\, 
0 \leq \lambda_k, \ \sum \lambda_k=1.
$$
This is exactly the statemant we wanted to prove, because $E_k \in \partial \mc{S}(\Kk)$ 
implies that $\mathrm{rank}(E_k)<\mathrm{dim}(\Kk)=\mathrm{rank}(D),$ that is, (\ref{rang}) 
holds, and $E_k \in \mc{L}_{D, \mathbf{A}}^{\Kk}$ implies that (\ref{linr}) holds. \par

Note that $D$ has a maximal rank in $\mc{S}(\Kk),$ hence it is a (relative) interior point of $\mc{S}(\Kk).$
On the other hand, $\mc{L}_{D, \mathbf{A}}^{\Kk}$ lies in the affine hull of $\mc{S}(\Kk)$ and has a positive dimension. Therefore, the intersection $\mc{L}_{D, \mathbf{A}}^{\Kk} \cap \mc{S}(\Kk)$ is not a single point.
\par

To show that the condition (\ref{tet}) is necessary as well, assume that
$$
\mathrm{dim}\left( \mathrm{span}\left\{ I^{\Kk}, A_1^{\Kk}, \dots, A_r^{\Kk}\right\}\right) = 
\left( \mathrm{dim}\, \Kk\right)^2
$$
for some  $\Kk \subset \C^n$ subspace with $ \mathrm{dim} \Kk > 1.$ Set $D \in \mc{S}(\Kk), \ 
\mathrm{rank}(D)>1.$
Because of the assumption $\mathrm{dim}\left( \mathrm{span}\left\{ I^{\Kk}, A_1^{\Kk}, \dots, 
A_r^{\Kk}\right\}\right) = \left( \mathrm{dim} \Kk\right)^2,$ $\mc{L}_{D, \mathbf{A}}^{\Kk}$ is a $0$ 
dimensional affine subspace of $\mc{B}^{sa}(\Kk),$ that is,
$\mc{L}_{D, \mathbf{A}}^{\Kk}=\{D\}.$ Therefore, we have by Lemma that the decomposition 
of $D$ is impossible. \qed
\end{proof}

The next example shows that for an arbitrary large $n$ there exists a set of self-adjoint 
operators with only three elements which is not variance-decomposable.

\begin{ex}
For every $n \geq 2$ we can show $A_1, A_2, A_3 \in \Mn$ self-adjoint matrices and a $D \in 
\mc{S}(\C^n)$ density with the following property.
If $P_1, \dots, P_m$ are rank-one projections such that  $D=\sum_{k=1}^{m} \lambda_k P_k$ with 
some $0 < \lambda_1, \dots, \lambda_m,$ $\sum_{k=1}^{m} \lambda_k =1,$ then
\be\label{nemegy}
\var_{D} (A_1, A_2, A_3) \ne \sum_{k=1}^{m} \lambda_k \var_{P_k} (A_1, A_2, A_3).
\ee
Let us use the Pauli matrices
$$
\sigma_1=\left[\ba{cc} 0 & 1 \\ 1 & 0 \ea \right], \quad
\sigma_2=\left[\ba{cc} 0 & -\im  \\ \im  & 0 \ea \right], \quad
\sigma_3=\left[\ba{cc} 1 & 0 \\ 0 & -1 \ea \right]
$$
to define $A_1, A_2, A_3 $ in block-matrices in the following way
$$
A_1:=\left[\ba{cc} \sigma_1 & 0 \\ 0 & 0 \ea\right], \quad
A_2:=\left[\ba{cc} \sigma_2 & 0 \\ 0 & 0 \ea\right], \quad
A_3:=\left[\ba{cc} \sigma_3 & 0 \\ 0 & 0 \ea\right],
$$
and $D:=\diag\left(\frac{1}{2},\frac{1}{2},0 ,\dots, 0 \right).$
By the Lemma, 
$$
\var_{D} (A_1, A_2, A_3) = \sum_{k=1}^{m} \lambda_k \var_{P_k} 
(A_1, A, A_3)
$$ 
if and only if $\tr P_k A_j=0$ for every $k$ and $j,$ but in this case we have $P_{k}^{\Kk}=0$ 
for $\Kk=\mathrm{range}(D),$ hence $D$ can not be a convex combination of the $P_k$'s. 
Therefore, (\ref{nemegy}) holds. \qed
\end{ex}

The proof of the statement of the previous example is shorter if we use the Theorem.
The only thing we have to observe is that
\[
\mathrm{dim}\left( \mathrm{span}\left\{ I^{\Kk}, A_1^{\Kk}, A_2^{\Kk}, A_3^{\Kk} \right\}\right) = 
\left( \mathrm{dim} \Kk\right)^2
\]
for $\Kk=\mathrm{range}(D).$

{\bf Acknowledgement.} This work was partially supported by the Hungarian Research Grant 
OTKA K104206 and we are grateful to Prof. G. T\'oth for communication. The authors wish to 
thank the referees for their constructive comments and advices.

\end{document}